\documentclass[12pt,a4paper,oneside]{amsart}
\usepackage{amsfonts, amsmath, amssymb, amsthm}
\usepackage[colorlinks=true,citecolor=blue]{hyperref}
\usepackage[margin=1.4in]{geometry}
\usepackage{txfonts}
\usepackage[T1]{fontenc}
\usepackage{bbm}
\usepackage{mathtools}
\usepackage{graphicx}
\usepackage{enumerate}
\usepackage{verbatim}
\usepackage{tikz}

\begingroup
    \makeatletter
    \@for\theoremstyle:=definition,remark,plain\do{%
        \expandafter\g@addto@macro\csname th@\theoremstyle\endcsname{%
            \addtolength\thm@preskip\parskip
            }%
        }
\endgroup

\newtheorem{theorem}{Theorem}
\newtheorem*{theorem*}{Theorem}
\newtheorem{lemma}[theorem]{Lemma}
\newtheorem{prop}[theorem]{Proposition}

\newtheorem*{claim*}{Claim}
\newtheorem*{quest*}{Question}
\newtheorem{claim}[theorem]{Claim}

\theoremstyle{definition}

\newtheorem*{remark*}{Remark}

\newtheorem{obs}{Observation}

\newcommand{\defeq}{\stackrel{_\text{\tiny{def}}}{=}}
\newcommand{\cal}[1]{\mathcal{#1}}

\begin{document}

\title{No weak $\varepsilon$-nets for lines and convex sets in space}
\author{Otfried~Cheong \and Xavier~Goaoc \and Andreas~F.~Holmsen}

\date{\today}




\begin{abstract}
  We prove that there exist no weak $\varepsilon$-nets of constant
  size for lines and convex sets in $\mathbb{R}^d$.
\end{abstract}

\maketitle 

\section{Introduction}

One of the most intriguing properties concerning families of convex
sets in $\mathbb{R}^d$ is the existence of weak $\varepsilon$-nets of
constant size. This was discovered by Alon et al.~\cite{weak-nets} in
the early 1990's.

\begin{theorem*}[Weak $\varepsilon$-net theorem]
For every $\varepsilon>0$ and every integer $d>0$ there exists an
integer $k = k(\varepsilon, d)$ with the following property. For any
finite set $X\subset \mathbb{R}^d$ there exists a set $T$ of $k$
points (not necessarily in $X$) that intersects any convex set that
contains at least $\varepsilon|X|$ points of $X$.
\end{theorem*}

The weak $\varepsilon$-net theorem plays a fundamental role in modern
combinatorial geometry, for instance in the proof of the celebrated
$(p,q)$-theorem \cite{pq-alon}, and determining the correct
growth-rate of the function $k(\varepsilon, d)$ is recognized as one
of the most important open problems in the area \cite{akmm}. The best
known lower bound of
$\Omega\big(\varepsilon^{-1}\log^{d-1}\varepsilon^{-1}\big)$ is due
to Bukh et al.~\cite{bukh}, while the upper bound
$O(\varepsilon^{-(d-1/2-\delta)})$---where $\delta$ is an arbitrarily
small constant---is a recent breakthrough due to
Rubin~\cite{rubin2021stronger}.

Another direction of active research has been to try to extend the
weak $\varepsilon$-net theorem to other types of geometric or
combinatorial set systems. For instance Alon and
Kalai~\cite{hyperplanes} established a weak $\varepsilon$-net theorem
for hyperplane transversals and convex sets in $\mathbb{R}^d$, while
Holmsen and Lee~\cite{absconv} established a weak $\varepsilon$-net
theorem in abstract convexity spaces (extending an earlier result of
Moran and \mbox{Yehudayoff}~\cite{moran2020weak}).

\bigskip

Recently, Imre B{\'a}r{\'a}ny (personal communication, see also
\cite[Conjecture 7.8]{BK}) conjectured a generalization of the weak
$\varepsilon$-net theorem for lines and convex sets
in~$\mathbb{R}^3$. The purpose of this note is to give a
counter-example to his conjecture.

\begin{theorem}\label{t:noepsnet}
For every $d \ge 3$, for every $0 < \varepsilon < 1$, and for every
integer~$k$, there exists a finite family $F$ of convex sets and a
finite set of lines~$L$ in general position in~$\mathbb{R}^d$ with the
following properties:
\begin{enumerate}
    \item\label{large intersection} Every member $K\in F$ intersects
      at least $\varepsilon |L|$ of the lines in $L$. 
    \item\label{large piercing} Any set of $k$ lines in $\mathbb{R}^d$
      misses at least one member $K\in F$. 
\end{enumerate}
\end{theorem}

\section{A construction in special position}

The first ingredient of our proof of Theorem~\ref{t:noepsnet} is the
following configuration of lines in special position in
$\mathbb{R}^3$:

\begin{lemma}\label{l:stab}
  For every integer~$n$ there is a family~$L_n$ of~$n$ lines in $\mathbb{R}^3$
  with the following property. For any subfamily~$B \subseteq L_n$ and
  any finite family of lines~$R$ with $B \cap R = \emptyset$ there
  exists a compact convex set that intersects every line of~$B$, but
  is disjoint from every line of~$R$.
\end{lemma}
\begin{proof}
  To construct the set of lines~$L_n$, we use the hyperbolic
  paraboloid~$\Sigma$ defined by the equation~$z = xy$.  For $\alpha
  \in \mathbb{R}$, let $\lambda_\alpha$ denote the line $\Sigma \cap
  \{x = \alpha\}$, or, in other words, the line~$(\alpha, t, \alpha
  t)$.  Symmetrically, for $\beta \in \mathbb{R}$, let $\ell_\beta$
  denote the line $\Sigma \cap \{y = \beta\}$, or, in other words, the
  line~$(t, \beta, \beta t)$.  Note that the lines~$\lambda_\alpha$
  and~$\ell_\beta$ form the two families of rulings of the quadratic
  surface~$\Sigma$.

We pick a set~$A$ of~$n$ positive numbers, and take~$L$ to be the set
of lines~$\{\lambda_{\alpha} \colon \alpha \in A\}$.

Now, let~$B \subset L$ and~$R$ a set of lines with~$R \cap B =
\emptyset$. We can write~$B = \{\lambda_{\alpha_1}, \dots,
\lambda_{\alpha_b}\}$ with~$0 < \alpha_1 < \alpha_2 < \dots <
\alpha_b$ and $b=|B|$.  Let~$R_\Sigma \subseteq R$ be the set of lines of the form
$\lambda_\alpha$, if there is any, and let~$R' \defeq R \setminus
R_\Sigma$.

We note that a line in~$R'$ is either of the form~$\ell_\beta$, for
some~$\beta \in \mathbb{R}$, or intersects~$\Sigma$ in at most two
points. In either case, a line in~$R'$ contains points of~$\Sigma$
with at most two distinct~$y$-coordinates. Since~$R'$ is finite, we
can choose some~$\beta^* \in \mathbb{R}$ to be a value such that no line
in~$R'$ contains a point on~$\ell_{\beta^*}$.

For $s\in \mathbb{R}$, let~$\Pi_s$ denote the plane with equation~$z =
\beta^* x + s$. The intersection~$\Sigma \cap \Pi_s$ is the hyperbola
$y = \beta^* + \frac{s}{x}$.  For $\alpha > 0$, let $p_\alpha(s)
\defeq (\alpha, \beta^* + \frac{s}\alpha, \beta^*\alpha + s)$ denote
the point $\Pi_s \cap \lambda_\alpha$.

Now let~$C(s)$ denote the convex hull of $\{p_{\alpha_i}(s) \colon 1
\leq i \leq b\}$. Note that~$C(s)$ is contained in the plane~$\Pi_s$,
and that the points~$p_{\alpha_i}(s)$ are convexly independent since
they all lie on the same branch of the hyperbola~$y = \beta^* +
\frac{s}{x}$. It follows that~$C(s)$ is a convex~$b$-gon that
intersects the hyperbola (and therefore the surface~$\Sigma$) only in
the points~$p_{\alpha_1}(s), p_{\alpha_2}(s), \dots, p_{\alpha_b}(s)$.
Since~$R_\Sigma \cap B = \emptyset$, no line in~$R_\Sigma$
intersects~$C(s)$ for any~$s > 0$.

We observe next that the distance between the point~$p_\alpha(s)$ and
the line~$\ell_{\beta^*}$ is at most the distance between~$p_\alpha(s)
= (\alpha, \beta^* + \frac{s}{\alpha}, \beta^*\alpha + s)$ and the
point~$(\alpha, \beta^*, \alpha\beta^*)$ on~$\ell_{\beta^*}$, and is
therefore at most~$s(1 + \frac{1}{\alpha})$. It follows that the
entire convex set~$C(s)$ is contained in a cylinder centered around
the line~$\ell_{\beta^*}$ of radius~$s\big(1+\frac1{\alpha_1}\big)$.

By construction, no line in~$R'$ intersects~$\ell_{\beta^*}$. If we
set~$\delta > 0$ to the minimum distance between~$\ell_{\beta^*}$ and
the lines of~$R'$ and choose~$s < \frac{\delta}{1 + 1/\alpha_1}$, then
no line of~$R'$ can intersect~$C(s)$, completing the proof.
\end{proof}

\section{A compactness argument}

The second ingredient of our proof of Theorem~\ref{t:noepsnet} is the
following compactness result:

\begin{lemma}\label{l:finiteness}
  Let $\hat{F}$ be a family of compact convex sets in $\mathbb{R}^3$
  such that any set of $k$ lines misses at least one member of
  $\hat{F}$. Then there exists a finite subfamily $F\subset \hat{F}$
  such that any set of $k$ lines misses at least one member of $F$.
\end{lemma}

We show this by a standard application of the De Bruijn--Erd{\H o}s
compactness principle (see for instance
\cite[Section~1.5]{GRS}). Specifically, let $H = (V,E)$ be a
hypergraph with vertex set $V$ and edge set $E$. Recall that a proper
$k$-coloring of $H$ is a partition $V = V_1\cup \cdots \cup V_k$ such
that no edge $e\in E$ is contained in any one of the parts $V_i$. (In
other words, a $k$-coloring is a partition of the vertices into $k$
independent sets.) For a subset $W\subset V$ let $H[W]$ denote the
induced hypergraph with vertex set $W$ and edge set $\{e\in E \colon
e\subset W\}$.

\begin{theorem*}[Compactness Principle]\label{compact}
Suppose all the edges of a hypergraph $H=(V,E)$ are finite. If every
finite set $W\subset V$ induces a hypergraph $H[W]$ which is
$k$-colorable, then $H$ is $k$-colorable.
\end{theorem*}

\begin{proof}[Proof of Lemma~\ref{l:finiteness}]
We construct a hypergraph $H$ whose vertices are the members of
$\hat{F}$ and whose edges are the minimal subfamilies of $\hat{F}$
which do not have a line transversal. In other words, a subfamily
$G\subset \hat{F}$ is an edge of $H$ if and only if
\begin{itemize}
\item There is no line transversal to the sets in $G$.  
\item Every proper subfamily of $G$ has a line transversal.
\end{itemize}

We note that a subfamily $G\subset \hat{F}$ is independent in $H$ if
and only if the members of $G$ have a common line transversal. In particular,
by the hypothesis, $H$ is \emph{not} $k$-colorable.

If $K$ is a compact convex set in $\mathbb{R}^3$, then the set of
lines in $\mathbb{R}^3$ that intersect $K$ is also
compact. Consequently, if an infinite family of compact convex sets
does not have a line transversal, then it contains some finite
subfamily that does not have a line transversal (by the finite
intersection property). Therefore the edges of $H$ are finite sets,
and so the Compactness Principle applies. Thus there exists a
\emph{finite} subfamily $F\subset \hat{F}$ such that the induced
hypergraph $H[F]$ is not $k$-colorable. In other words, any set of $k$
lines misses at least one member of the finite family~$F$.
\end{proof}

\section{Proof of Theorem~\ref{t:noepsnet} in three dimensions}

Given $0< \varepsilon < 1$ and $k$, choose an integer $n >
\frac{k}{1-\varepsilon}$, and let $L_n$ denote the family of lines
from Lemma~\ref{l:stab}. Let $\hat{F}$ be the family of all compact
convex sets in $\mathbb{R}^3$ that intersect at least $\varepsilon n$
of the lines in $L_n$. It follows from Lemma~\ref{l:stab} that if $R$
is a finite family of lines that intersects any member of $\hat{F}$,
then $R$ must contain at least one line from every subfamily of $L_n$
of size at least $\varepsilon n$. Therefore $|R| > (1-\varepsilon)n >
k$, and consequently, any set of $k$ lines misses at least one member of
$\hat F$.

By Lemma~\ref{l:finiteness}, there is a finite subfamily $F\subset
\hat{F}$ such that any set of $k$ lines misses at least one member of
$F$. We now have a finite family $F$ of compact convex sets and a
finite family of lines $L_n$ which satisfy properties~\eqref{large
  intersection} and~\eqref{large piercing} of
Theorem~\ref{t:noepsnet}. The only thing missing is that the lines
in~$L_n$ are not in general position.

Since $F$ is finite and its members are compact, their union, $\cup
F$, is contained in some ball $B\subset \mathbb{R}^3$. Let $X$ denote
the set of $k$-tuples of lines (not necessarily distinct) that
intersect $B$. Note that $X$ is also compact. For each $x\in X$, let
$f(x)$ denote the maximum among the distances between members of $F$
and lines in~$x$. Note that $f:X \to \mathbb{R}$ is continuous and
that $f(x)>0$ for every $x \in X$. Since $X$ is compact, $f$ attains
its minimum $\delta >0$ over $X$. This means that if we replace each
member $K\in F$ by its Minkowski sum with a ball of radius $\delta' <
\delta/2$, then the resulting family $F_{\delta'}$ also satisfies
properties~\eqref{large intersection} and~\eqref{large piercing}
with~$L = L_n$. We also observe that by choosing $\delta'>0$ sufficiently small, we can guarantee that every line of $L_n$ that intersects $K \in F$ will 
intersect the {\em interior} of the inflation of $K$, and any line of
$L_n$ that misses $K \in F$ also misses the
inflation. 
  Again using that the family $F$ is finite, it follows
that each of the lines in $L_n$ can be perturbed slightly while
property \eqref{large intersection} remains true. This concludes the
proof of Theorem~\ref{t:noepsnet} for $d=3$.

\section{Proof of Theorem \ref{t:noepsnet} in higher dimension}

Assume that $d \ge 4$ and let us fix some 3-dimensional subspace $S
\subset \mathbb{R}^d$. As above, given $0< \varepsilon < 1$ and $k$,
there exists a finite family $L$ of lines in $S$ and a finite family
$F$ of compact convex sets in $S$ that satisfy property~\eqref{large
  intersection}. Moreover, any $k$ lines in $S$ misses at least one
member of $F$. Observe that if a line $\ell$ in $\mathbb{R}^d$
intersects some $K \in F$, then the orthogonal projection of $\ell$ on
$S$ also intersects $K$. It follows that $L$ and $F$ also satisfy
property~\eqref{large piercing}.

The set $F$ is contained in a closed ball $B \subset \mathbb{R}^d$ and
the set $X$ of $k$-tuples of lines (not necessarily distinct) of
$\mathbb{R}^d$ that intersect $B$ is compact. We can therefore inflate the
elements of $F$ and perturb $L$ into general position in $\mathbb{R}^d$ as above.

\section{A related question} \label{related}

The key step to our proof of Theorem \ref{t:noepsnet} is the
construction given in Lemma \ref{l:stab}. While the lines in that
construction are carefully chosen, it is natural to ask whether there
is a more general underlying property at work. The following question
arises:

\begin{quest*}
  Given a finite set of blue and red lines in $\mathbb{R}^3$, does
  there exist a convex set that intersects all blue lines and avoids
  all red lines?
\end{quest*}

Some general position assumption is needed, since the answer is
negative for any configuration with three parallel lines, one red
between two blue. We also note that this question is closely
related to the convexity structure investigated by Goodman and Pollack
in \cite{goodman}.

Let $b$ denote the number of blue lines and $r$ the number of red
lines, and assume their union is a set of lines in general position in
$\mathbb{R}^3$. It is quite easy to see that for certain values of $b$
and $r$ we get a positive answer for the question above. For instance,
the case when $b$ is arbitrary and $r=1$ is dealt with in
\cite{goodman} (which additionally deals with the more general
question for $k$-flats in $\mathbb{R}^d$). More generally, we get a
positive answer for lines in general position whenever $b\leq 3$ or
$r\leq 2$. (We leave the simple proofs to the reader.)

\medskip

On the other hand, there are general examples for which the answer is negative. In particular we have the following:

\begin{prop} \label{claim:redblue}
  There exists a family of 9 blue lines and 13 red lines in
  $\mathbb{R}^3$ such that any convex set that intersects the blue
  lines must intersect at least one of the red lines. Moreover, the
  construction is ``stable" in the sense that the lines can be
  slightly perturbed without affecting the intersection properties.
\end{prop}

We only give an informal sketch of the construction, and leave a
detailed verification to the reader.  First we define the following
three {\em blue} lines
\[
\ell_x = (t,1,-1), \;\; \ell_y = (-1,t,1), \;\; \ell_z = (1,-1,t),   \;\;
t\in \mathbb{R}.
\]
Note that these three lines contain disjoint edges of the unit cube $C
= [-1,1]^3$.

Next, we will define four \emph{red} lines. For each pair of opposite
vertices of the cube~$C$ choose a \emph{red} line whose distance to
these two opposite vertices is very small. In other words, we choose
four \emph{red} lines that are very close to the main diagonals of
$C$.  (See \textsc{Figure}~\ref{cube})

Now consider a triangle $T$ with one vertex on each of the lines
$\ell_x, \ell_y, \ell_z$. The crucial observation is:
\begin{quote} \label{quote}
  \emph{If no vertex of~$T$ lies in the cube~$2C = [-2,2]^{3}$, then~$T$
    intersects one of the red lines.}
\end{quote}

\begin{figure}[h]
    \centerline{\includegraphics{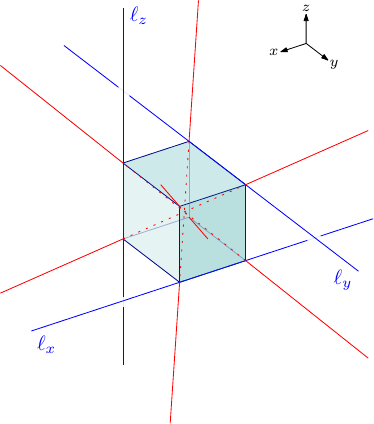}}
    \caption{\small The three \emph{blue} lines $\ell_x, \ell_y,
      \ell_z$ inscribing the cube $C$ and four \emph{red} lines close
      to the main diagonals of $C$.}
    \label{cube}
\end{figure}

The observation above will be proved in Claim \ref{claim:diag} below. The idea of the proof is to analyze the orthogonal projections of our line configuration onto planes perpendicular to the main diagonals of the cube $C$. This reduces the problem to a property of planar geometry, stated in  Claim \ref{obs:ray} below. 

\medskip

Before proving the observation above, let us use it to complete the construction of Proposition \ref{claim:redblue}. Consider three cubes~$C_1, C_2, C_3$ in
$\mathbb{R}^3$, all congruent to $C$, such that the centers of the
$C_i$ are far apart and form an equilateral triangle~$\Delta$. Each
cube~$C_i$ has three \emph{blue} lines containing disjoint edges
of~$C_i$, and four \emph{red} lines as chosen as above. The thirteenth
\emph{red} line passes through the barycenter of the triangle~$\Delta$
and is orthogonal to the plane containing~$\Delta$.  There is enough
freedom to rotate the cubes such that the union of all red and blue
lines is in general position.

We claim that any convex set~$K$ that intersects the 9 \emph{blue}
lines must be intersected by one of the \emph{red} lines. The reason
is the following: $K$ contains a point on each of the~9 blue
lines. If, for some~$C_i$, none of these points lies in~$2C_i$, then
by the observation above~$K$ intersects one of the perturbed main
diagonals of~$C_i$.  Otherwise~$K$ contains a point in~$2C_i$, for all
three~$C_i$.  But then it must intersect the thirteenth red line.

\bigskip

It remains to prove the observation regarding the cube $C$ and the {\em red} lines close to its main diagonals. 

Let $r_1, r_2, r_3$ be three rays in $\mathbb{R}^2$. Assume that no two of the rays are parallel to a common line, and that no line intersects all three rays. Any three rays satisfying these conditions will be called a {\em separated triple} of rays. Our first goal is to describe the set of points (if there are any) that intersect every triangle spanned by $r_1, r_2, r_3$, that is, the triangles with one vertex on each of the rays.

Let $X$ denote the triangle spanned by the initial points of the rays $r_1, r_2, r_3$.
Note that since the rays form a separated triple, each ray $r_i$ is disjoint from a unique side of the triangle $X$ which we denote by $s_i$. 
Define $X_i$ as the set of points that lie on a translate of $r_i$ which starts at some point on $s_i$. 
Thus $X_i$ is a half-strip emanating from the side $s_i$. 
At last, define the {\em joint region} of the rays $r_1, r_2, r_3$ as the intersection $X_1\cap X_2\cap X_3$. Note that it is possible for the joint region to be empty. (See \textsc{Figure} \ref{fig:rays})

\begin{claim}\label{obs:ray}
 Let $r_1, r_2, r_3$ be a separated triple of rays. If a point $x$ is in the joint region of $r_1, r_2, r_3$, then $x$ is contained in every triangle spanned by $r_1, r_2, r_3$. Moreover, if $x$ is in the interior of the joint region of $r_1, r_2, r_3$, then $x$ is also in the interior of the joint region of $r_1', r_2', r_3'$, where $r_i'$ is any sufficiently small perturbation of the ray $r_i$.
\end{claim}
\begin{proof} Let $X$ be the triangle spanned by the initial points of the rays, and let $Y$ be their joint region. 
We first observe that $Y\subseteq X$. To see this, assume there exists $x\in Y\setminus X$. Then there is a side of $X$, say $s_1$, such that the line containing $s_1$  separates $x$ from $X$. But since $x\in X_1$, this implies that $r_1$ intersects the line containing $s_1$. The endpoints of $s_1$ are the initial points of the rays $r_2$ and $r_3$, 
and so the line containing $s_1$ intersects all three rays, which contradicts the assumption that the rays form a separated triple. 
\begin{figure}
\centerline{\includegraphics{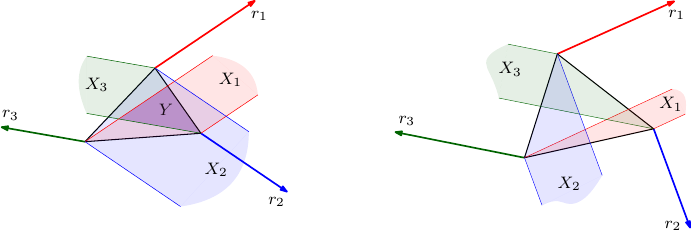}}
\caption{\small Left: A triple of separated rays with non-empty joint region $Y=X_1\cap X_2\cap X_3$. Right: An example with empty joint region.}\label{fig:rays}
\end{figure}

\smallskip

We now prove the first part of 
Claim~\ref{obs:ray}. Let $x$ be an arbitrary point in $Y$ and suppose there is a triangle $T$ spanned by $r_1, r_2, r_3$ which does not contain $x$. There is a line $\ell$, which we may assume is horizontal, such that $T$ lies above $\ell$, while $x$ lies below $\ell$. Next we consider where the triangle $X$ is located relative to the line $\ell$. Up to relabelling of the rays and sides of $X$, and possibly shifting $\ell$ slightly up or down, we may assume that the side $s_1$ of $X$ either lies in the upper open half-plane or in the lower open half-plane bounded by $\ell$.

If $s_1$ lies above $\ell$, then the initial point of the ray $r_1$ must lie below $\ell$, or else $x$ is not contained in $X$. Therefore $r_1$ starts in the lower half-plane and then enters the upper half-plane. But this implies that the half-strip $X_1$ is contained in the upper open half-plane and can not contain the point $x$. 

If $s_1$ lies below $\ell$, then the rays $r_2$ and $r_3$ both start in the lower half-plane and then enter the upper half-plane. This implies that any horizontal line in the upper half-plane intersects both rays $r_2$ and $r_3$. In particular, the vertex of $T$ which is contained in $r_1$ lies on such a horizontal line, and therefore there is a line intersecting all three rays, contradicting the assumption that they form a separated triple. 

\smallskip

Let us now prove the second part of 
Claim~\ref{obs:ray}. Suppose $x$ is in the interior of the joint region of $r_1, r_2, r_3$. There exists an $\varepsilon > 0$ such that the distance from $x$ to the boundary of $X_i$ is at least $\varepsilon$ for every $i$. Thus, if we make the perturbation of each of the rays sufficiently small, the distance from the boundary of $X_i$ to $x$ will decrease by strictly less than $\varepsilon$.
\end{proof}

\medskip

We return to the three-dimensional situation with the three {\em blue} lines $\ell_x, \ell_y, \ell_z$ inscribing the cube $C$ and the four {\em red} lines that are very close to the main diagonals of $C$.
For a point $(x_1,x_2,x_3) \in \mathbb{R}^3$ we associate the triangle
$T(x_1,x_2,x_3)$ as the convex hull of $\{v_1,v_2,v_3\}$ where $v_1 =
(x_1,1,-1) \in \ell_x$, $v_2 = (-1,x_2,1) \in \ell_y$, and $v_3 =
(1,-1,x_3) \in \ell_z$.

Let $\ell$ be a line through the origin in~$\mathbb{R}^3$ with unit
direction vector~$u$. For $\varepsilon>0$, we say that a line~$\ell'$
with unit direction vector~$u'$ is \emph{an $\varepsilon$-perturbation
  of~$\ell$} if the distance between $\ell'$ and the origin is less
than $\varepsilon$ and $u\cdot u' > 1-\varepsilon$.

\begin{claim} \label{claim:diag}
  There exists an $\varepsilon >0$ with the following property: let
  $m_1$, $\dots$, $m_4$ be arbitrary $\varepsilon$-perturbations of
  the lines
  \[
  (t, t, t), \;\; (t, t, -t), \;\; (t, -t, t), \;\; (-t, t, t), \;\;
  t \in \mathbb{R},
  \]
  respectively.  If~$|x_i| \geq 2$ for~$i \in \{ 1, 2, 3\}$, then one of
  the lines~$m_1,\dots, m_4$ intersects the interior
  of~$T(x_1,x_2,x_3)$.
\end{claim}

\begin{proof} We start by noting that among the four main diagonals, only~$(t,t,t)$ is  disjoint from each of the lines $\ell_x, \ell_y, \ell_z$. The remaining three diagonals each intersect exactly two of the lines $\ell_x, \ell_y, \ell_z$, and it turns out that these three cases are symmetric up to reflective symmetry. So below there are only two cases to consider. 

In the arguments that follow, for a unit vector $u\in S^2$, let $\pi_u : \mathbb{R}^3 \to u^\perp$ denote the orthogonal projection onto the orthogonal complement $u^\perp$. 

\smallskip

First we consider the diagonal $(t,t,t)$. Define the following rays in $\mathbb{R}^3$:
\[\begin{array}{ccclccccl}
    R_1 & = & (x_1,1,-1), & x_1 \geq 2 & \hspace{1cm} &     Q_1 & = & (x_1,1,-1), & x_1 \leq 2 \\
    R_2 & = & (-1,x_2,1), & x_2 \geq 2 &&     Q_2 & = & (-1,x_2,1), & x_2 \leq 2 \\
    R_3 & = & (1,-1,x_3), & x_3 \geq 2 &&    Q_3 & = & (1,-1,x_3), & x_3 \leq 2.
\end{array}\]

Let $u = \left(\frac{1}{\sqrt{3}}, \frac{1}{\sqrt{3}}, \frac{1}{\sqrt{3}}\right)$ and define the rays $r_i = \pi_u(R_i)$ and $q_i = \pi_u(Q_i)$. We claim that each of the triples $r_1, r_2, r_3$ and $q_1,q_2,q_3$ are separated and that the origin lies in the interior of each of their joint regions. This can be seen from \textsc{Figure \ref{fig:projection}} (left). 
  \begin{figure}
    \centerline{\includegraphics{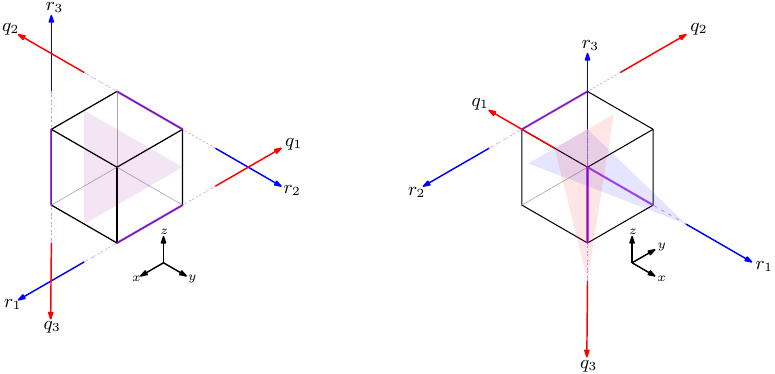}}
    \caption{\small Left: Orthogonal projection along the main
      diagonal~$(t, t, t)$. The shaded triangle is the joint region of $r_1, r_2, r_3$ and of $q_1, q_2, q_3$. Right: Orthogonal projection along the main
      diagonal~$(t, -t, t)$. The blue shaded area is the joint region of $r_1, r_2, r_3$ and the red shaded region is the joint region of $q_1, q_2, q_3$.}
    \label{fig:projection}
  \end{figure}

By the first part of Claim \ref{obs:ray} it follows that any triangle spanned by the rays $r_1, r_2, r_3$ or by the rays $q_1, q_2, q_3$ will contain every point in a small neighborhood of the origin. 
Therefore a line in the direction of $u$ which is sufficiently close to the origin will intersect any triangle spanned by the rays $R_1, R_2, R_3$ or by the rays $Q_1, Q_2, Q_3$. Now consider a unit vector $u'$ which is very close to $u$ and define the rays $r'_i = \pi_{u'}(R_i)$ and $q_i' = \pi_{u'}(Q_i)$. As long as $u'$ is sufficiently close to $u$, the rays $r_i'$ and $q_i'$ will be sufficiently small perturbations of the rays $r_i$ and $q_i$, respectively. So by the second part of Claim \ref{obs:ray} the origin is in the interior of the joint regions of the $r_i'$ and of the $q_i'$. Therefore any line in the direction $u'$ which is sufficiently close to the origin will intersect every triangle spanned by the rays $R_1, R_2, R_3$ or by the rays $Q_1, Q_2, Q_3$.

\smallskip

Now we consider the diagonal $(t,-t,t)$. Define the following rays in $\mathbb{R}^3$:
\[\begin{array}{ccclccccl}
    R_1 & = & (x_1,1,-1), & x_1 \geq 2 & \hspace{1cm} &     Q_1 & = & (x_1,1,-1), & x_1 \leq 2 \\
    R_2 & = & (-1,x_2,1), & x_2 \leq 2 &&     Q_2 & = & (-1,x_2,1), & x_2 \geq 2 \\
    R_3 & = & (1,-1,x_3), & x_3 \geq 2 &&    Q_3 & = & (1,-1,x_3), & x_3 \leq 2.
\end{array}\]

Let $u = \left(\frac{1}{\sqrt{3}}, -\frac{1}{\sqrt{3}}, \frac{1}{\sqrt{3}}\right)$ and define the rays $r_i = \pi_u(R_i)$ and $q_i = \pi_u(Q_i)$. From \textsc{Figure} \ref{fig:projection} (right) we see that each of the triples $r_1, r_2, r_3$ and $q_1, q_2, q_3$ are separated and the origin lies in the interior of each of their joint regions. The rest of the argument is the same as above.
\end{proof}

\section{Final remarks}

\begin{enumerate}[(i)]
\item We do not know if an analogue of Theorem~\ref{t:noepsnet} holds
  for $2$-planes in $\mathbb{R}^4$.

\item In the proof of Lemma \ref{l:stab} it is sufficient to consider
  parameters $\beta^*\in \mathbb{N}$ and $s\in \big\{\frac{1}{i}
  \colon i\in \mathbb{N} \big\}$ in the definition of the line
  $\ell_{\beta^*}$ and the planes $\Pi_s$. This in turn implies that
  when we apply the Compactness Principle we may take $\hat{F}$ to be
  a {\em countable} family of compact convex sets, which means that
  the Axiom of Choice does not need to be invoked. (See the first
  proof of Theorem 4 in \cite{GRS}.)
\end{enumerate}

\bibliographystyle{plain} \bibliography{ref}

\begin{thebibliography}{10}

\bibitem{weak-nets}
Noga Alon, Imre B{\'a}r{\'a}ny, Zolt{\'a}n F{\"u}redi, and Daniel~J. Kleitman.
\newblock Point selections and weak $\varepsilon$-nets for convex hulls.
\newblock {\em Combinatorics, Probability and Computing}, 1(3):189--200, 1992.

\bibitem{hyperplanes}
Noga Alon and Gil Kalai.
\newblock Bounding the piercing number.
\newblock {\em Discrete \& Computational Geometry}, 13(3):245--256, 1995.

\bibitem{akmm}
Noga Alon, Gil Kalai, Ji{\v r}{\' i} Matou{\v s}ek, and Roy Meshulam.
\newblock Transversal numbers for hypergraphs arising in geometry.
\newblock {\em Advances in Applied Mathematics}, 29(1):79--101, 2002.

\bibitem{pq-alon}
Noga Alon and Daniel~J. Kleitman.
\newblock Piercing convex sets and the {Hadwiger-Debrunner} (p, q)-problem.
\newblock {\em Advances in Mathematics}, 96(1):103--112, 1992.

\bibitem{BK}
Imre B{\'a}r{\'a}ny and Gil Kalai.
\newblock Helly-type problems.
\newblock {\em arXiv preprint arXiv:2108.08804}, 2021.

\bibitem{bukh}
Boris Bukh, Ji{\v{r}}{\'\i} Matou{\v{s}}ek, and Gabriel Nivasch.
\newblock Lower bounds for weak epsilon-nets and stair-convexity.
\newblock {\em Israel Journal of Mathematics}, 182(1):199--228, 2011.

\bibitem{goodman}
Jacob~E. Goodman and Richard Pollack.
\newblock Foundations of a theory of convexity on affine grassmann manifolds.
\newblock {\em Mathematika}, 42(2):305–328, 1995.

\bibitem{GRS}
Ronald~L. Graham, Bruce~L. Rothschild, and Joel~H. Spencer.
\newblock {\em Ramsey theory}, volume~20.
\newblock John Wiley \& Sons, 1991.
\newblock 2nd edition.

\bibitem{absconv}
Andreas~F. Holmsen and Donggyu Lee.
\newblock {Radon numbers and the fractional Helly theorem}.
\newblock {\em Israel Journal of Mathematics}, 241(1):433--447, 2021.

\bibitem{moran2020weak}
Shay Moran and Amir Yehudayoff.
\newblock On weak $\varepsilon$-nets and the {R}adon number.
\newblock {\em Discrete \& Computational Geometry}, 64(4), 2020.

\bibitem{rubin2021stronger}
Natan Rubin.
\newblock Stronger bounds for weak $\varepsilon$-nets in higher dimensions.
\newblock In {\em Proceedings of the 53rd Annual ACM SIGACT Symposium on Theory
  of Computing}, pages 989--1002, 2021.

\end{thebibliography}

\end{document}